\documentclass{amsart}

\usepackage{amsmath}
\usepackage{amsthm}
\usepackage{amssymb}
\usepackage{amsbsy}
\usepackage{amsfonts}
\usepackage{amstext}
\usepackage{amscd}
\usepackage{tikz}
\usepackage{graphicx}

\usepackage{color}

\numberwithin{equation}{section}
\theoremstyle{plain}
\newtheorem{thm}{Theorem}[section]
\newtheorem{prop}[thm]{Proposition}
\newtheorem{cor}[thm]{Corollary}
\newtheorem{lem}[thm]{Lemma}
\theoremstyle{definition}
\newtheorem{exa}[thm]{Example}

\newtheorem{rem}[thm]{Remark}
\newtheorem{defi}[thm]{Definition}

\theoremstyle{definition}

\begin{document}
\title[Cumulants for finite free convolution]
{Cumulants for finite free convolution}

\author{Octavio Arizmendi}
\author{Daniel Perales }

\email{octavius@cimat.mx, daniel.perales@cimat.mx}

\date{\today}
\begin{abstract}
In this paper we define cumulants for finite free convolution. We give a moment-cumulant formula and show that these cumulants satisfy desired properties: they are additive with respect to finite free convolution and they approach free cumulants as the dimension goes to infinity.

\end{abstract}

\maketitle

\section{Introduction}

Since the original paper of Voiculescu \cite{Voi91}, where he discovered \emph{asymptotic freeness}, Free Probability has given answers to many questions in random matrix theory in the  \emph{asymptotic regime}, see for example \cite{ANV, BSTV, BMS, BF, CDM, CDMFF, Sh}.  Recently, Marcus, Spielman, and Srivastava \cite{MSSd} found a connection between polynomial convolutions and addition of random matrices, which in the limit is related to free probability. The key idea is that instead of looking at distributions of eigenvalues of random matrices, one looks at the (expected) characteristic polynomial of a random matrix.

To be precise, for a matrix $M$, let $
\chi_{M}(x) = \det (xI - M)
$ be  the characteristic polynomial of the matrix $M$. For $d$-dimensional hermitian matrices  $A$ and $B$  with characteristic polynomials $p$ and $q$, respectively, one defines the \textit{finite free additive convolution} of $p$ and $q$ to be
\begin{equation*}
  p (x) \boxplus_{d} q (x) = \textbf{E}_Q [\chi_{A + Q B Q^{T}}(x)],
\end{equation*}
where the expectation is taken over orthonormal matrices $Q$ sampled according to the Haar measure. The convolution does not depend on the specific choice of $A$ and $B$, but only on $p$ and $q$.

The connection with free probability is that, as $d\to\infty$,  this polynomial convolution approximates  free additive convolution. The connection is quite remarkable since, as proved in \cite{MSSd}, this convolution has appeared before and there are very explicit formulas to calculate the coefficients of the finite free convolution of two polynomials. Indeed, for $p(x) = \sum_{i=0}^d x^{d-i} (-1)^i  a^p_i$ and  $q(x) = \sum_{i=0}^d x^{d-i} (-1)^i  a^q_i$, the finite free convolution of $p$ and $q$ is given by
\begin{equation*}  p (x) \boxplus_{d} q (x) = \sum_{k=0}^d x^{d-k} (-1)^k  \sum_{i+j=k} \frac{(d-i)!(d-j)!}{d! (d-i-j)!} a^p_i a^q_j.
\end{equation*}

A very successful approach to free probability is the combinatorial one developed by Speicher \cite{Sp94} based on free cumulants. He showed that the combinatorial structure of free cumulants is governed by non-crossing partitions. Hence, understanding the combinatorics on non-crossing partitions gave insight in many cases where analytic expressions could not be found.

 In this paper we want to give a similar combinatorial treatment to finite free additive convolution. Our main contribution is describing cumulants for finite free additive convolution, by deriving moment-cumulant formulas. We will give alternative proofs of some of the results presented in \cite{Mar}. Our approach has the advantage that it does not  involve using random matrices and avoids the analytic machinery. Hence, this approach gives a different, combinatorial, understanding of finite free additive convolution.  We also prove that these cumulants approximate free cumulants, when the degree of the polynomials tends to infinity.

Apart from this introduction, this paper is organized in five sections. In Section 2, we recall the preliminaries of non-crossing partitions, incidence algebras and free convolution. The main results are divided in Sections 3, 4 and 5. In Section 3 we introduce finite free cumulants and prove that they are additive. In Section 4 we give a moment-cumulant formula. In Section 5 we show that finite free cumulants approach free cumulants. Finally, in Section 6 we give some applications and examples.

\section{Preliminaries}

In this section we give the necessary preliminaries on incidence algebras, partitions, non crossing partitions, and free cumulants. For a detailed account of the combinatorial theory of free probability we refer the reader to the monograph of Nica and Speicher \cite{NiSp}.

\subsection{Posets}
We will be concerned with the set of partitions, seen as a poset and with functions in the incidence algebra of this poset. For details, the standard reference is \cite{Stan12}.

Incidence algebras were first introduced and systematically studied by Rota \cite{Rota} in order to generalize the Inclusion-Exclusion Principle. The incidence algebra $I(P,\mathbb{C})$ of a finite poset $(P,\leq)$ consists of all functions $f:P^{(2)}\rightarrow \mathbb{C}$ such that $f(\pi, \sigma )=0$ whenever $\pi \nleq \sigma $. We can also consider functions of one variable; these are restrictions of functions of two variables as above to the case where the first argument is equal to $0$, i.e. $f(\pi )=f(0,\pi )$ for $\pi \in P$.

We endow $I(P,\mathbb{C})$ with the usual structure of vector space over $\mathbb{C}.$ On this incidence algebra we have a canonical multiplication or (combinatorial) convolution 
 defined by

\begin{equation*}
(F\ast G)(\pi,\sigma ):=\sum\limits_{\substack{ \rho \in P  \\ \pi \leq \rho \leq \sigma 
}}F(\sigma, \rho )G(\rho ,\sigma )\text{.}
\end{equation*}
Also, for functions $f:P\rightarrow \mathbb{C}$ and $G:P^{(2)}\rightarrow 
\mathbb{C}$ we consider the convolution $f\ast G:P\rightarrow $ defined by
\begin{equation*}
(f\ast G)(\sigma ):=\sum\limits_{\substack{ \rho \in P \\ \rho \leq \sigma 
}}f(\rho )G(\rho ,\sigma )\text{.}
\end{equation*}

One of the most important properties of incidence algebras is the following inversion principle, which generalizes the well-known M\"obius inversion in $\mathbb{N}$. 
\begin{prop}[Principle of M\"obius inversion]
  \label{prop:moebius}
  On any poset $(P,\leq)$ there is a unique \emph{M\"obius function}
  $\mu:P\times P \to \mathbb{Z}$ such that for any pair of functions $f,g:P\to \mathbb{C}$
 the identity
  \begin{equation}
    \label{eq:fx=sumgy}
  f(x)= \sum_{y\leq x} g(y)
  \end{equation}
  holds for every $x\in P$ if and only if
  \begin{equation}
  g(x)= \sum_{y\leq x} f(y)\,\mu(y,x). 
  \end{equation}
  In particular, if, given $f$, two functions $g_1$ and $g_2$ satisfy
  \eqref{eq:fx=sumgy}, then $g_1$ and $g_2$ coincide.
\end{prop}

A chain is a totally ordered subset of a poset. We say that $P$ has a \emph{minimum} if there exists an element  $\hat 0\in P$ such that $t\geq \hat 0$ for all $t \in P$. Similarly, $P$ has a \emph{maximum} if there exists $\hat 1\in P$ such that $t \leq \hat 1$ for all $t\in P$.

A chain (or totally ordered set or linearly ordered set) is a poset in which any two elements are comparable. A subset $C$ of a poset $P$ is
called a chain if $C$ is a chain when regarded as a subposet of $P$. The chain $C$ of $P$ is called saturated
(or unrefinable) if there does not exist $u\in P - C$ such that $s < u < t$ for some $s, t \in C$
and such that $C \cup {u}$ is a chain. A poset $P$ with minimum $\hat 0$, is ranked if, for each $x\in P$, every saturated
$\hat 0-x$ chain has the same length. 

Given a ranked poset, we get a rank function
$rk: P\to \mathbb{N}$ defined by setting $rk(x)$ to be the length of a $\hat 0-x$ saturated chain.
We define the rank of a ranked poset $P$ to be $$rk(P)=\max_{x\in P} rk(x).$$

\begin{defi}[Characteristic polynomial of a ranked poset]
The characteristic polynomial of a poset $P$ with rank $rk: P\to \mathbb{N}$ is defined by 
$$\chi(P,t):=\sum_{\pi \in \mathcal{P}(k)} \mu (0, \pi) t^{rk(P)-rk(x)}.$$
\end{defi}

\subsection{Set Partitions}
Now, we specialize in partitions and non crossing partitions.

\begin{defi}

(1) We call $\pi =\{V_{1},...,V_{r}\}$ a \textbf{partition }of the set $[n]:=\{1, 2,\dots, n\}$
if and only if $V_{i}$ $(1\leq i\leq r)$ are pairwise disjoint, non-void
subsets of $S$, such that $V_{1}\cup V_{2}...\cup V_{r}=\{1, 2,\dots, n\}$. We call $
V_{1},V_{2},\dots,V_{r}$ the \textbf{blocks} of $\pi $. The number of blocks of 
$\pi $ is denoted by $\left\vert \pi \right\vert $.

(2) A partition $\pi =\{V_{1},...,V_{r}\}$ is called \textbf{non-crossing} if for all $1 \leq a < b < c < d \leq n$
if $a,c\in V_{i}$ then there is no other subset $V_{j}$ with $j\neq i$ containing $b$ and $d$.

\end{defi}

We will denote the set of partitions  of $[n]$ by $\mathcal{P}(n)$ and the set of non-crossing partitions of $[n]$ by $\mathcal{NC}(n)$.

Many statistics of partitions and non-crossing partitions are known. In particular, the following will be useful for us while comparing finite free cumulants with free cumulants.

\begin{lem}
\label{type}Let $r_{1},r_{2},...r_{n}$ be nonnegative integers such that $r_{1}+2r_{2}+\cdots+nr_{n}=n$. Denote by $p_r=r_1!r_2!\cdots r_n!$ and $r_{1}+r_{2}+\cdots+r_{n}=m.$  Then 
\begin{enumerate} \item The number of
partitions of $\pi $ in $\mathcal{NC}(n)$ with $r_{1}$ blocks of size $1$, $r_{2}$
blocks of size $2$ ,$\dots$, $r_{n}$ blocks of size $n$ equals
\begin{equation}
\frac{n!}{p_r(n-m+1)!}.
\label{rtype partitions}
\end{equation}
\item The number of partitions of $\pi $ in $\mathcal{P}(n)$ with $r_{1}$ blocks of size $1$, $r_{2}$
blocks of size $2$ ,$\dots$, $r_{n}$ blocks of size $n$ equals
\begin{equation}
\frac{n!}{p_r (2!)^{{r_{2}}}(3!)^{{r_{3}}}\cdots (n!)^{{r_{n}}}}.
\label{rtype partitions 2}
\end{equation}
\end{enumerate}
\end{lem}

The sets $\mathcal{P}(n)$ and $\mathcal{NC}(n)$ can be equipped with the partial order $\preceq$ of reverse refinement ($\pi\preceq\sigma$ if and only if every block of $\pi$ is completely contained in a block of $\sigma$), making them ranked posets with rank $n-|\pi|+1$. Thus one can consider the incidence algebras of $\mathcal{P}(n)$ and $\mathcal{NC}(n)$. In both cases the minimum is given by the partition with $n$ blocks, $0_n=\{\{1\},\{2\},\cdots, \{n\}\}$, and the maximum is given by the partition with $1$ block, $1_n=\{\{1,2,\cdots,n\}\}$.

For two partitions $\sigma,\rho $ in the set of partitions $\mathcal{P}(n)$ the M\"obius function is given by 
$$\mu (\sigma ,\rho)=(-1)^{{|\sigma|-|\rho|}}(2!)^{{r_{3}}}(3!)^{{r_{4}}}\cdots ((n-1)!)^{{r_{n}}},$$
where $r_i$ is the number of blocks of $\rho$ that contain exactly $i$ blocks of $\sigma$.
In particular, for $\sigma=0_n$ we have
\begin{equation}\label{Moeb1} \mu (0_n ,\rho )=(-1)^{n-{|\rho|}}(2!)^{{t_{3}}}(3!)^{{t_{4}}}\cdots ((n-1)!)^{{t_{n}}},\end{equation}
where $t_i$ is the number of blocks of $\rho$ of size $i$.

It is well known (see \cite{Rota}) that for the set of partitions  the characteristic polynomial is given by  the Pochhammer symbol,
 \begin{equation} \label{charpol} \chi(\mathcal{P}(n),t)=(t)_n=t(t-1)\cdots(t-n+1). \end{equation}

Given a sequence of complex numbers $f=\{f_n\}_{n\in \mathbb{N}}$ we may extend $f$ to partitions in a multiplicative way by the formula $$f_\pi=f_{|V_1|}f_{|V_2|}\cdots f_{|V_n|},$$
where $V_1,\dots,V_n$ are the blocks of $\pi.$ In this note we will  frequently use the multiplicative extensions of the Pochhammer sequence $(d)_n=(d)(d-1)\cdots (d-n+1)$ and the factorial sequence $n!$, whose extensions will be denoted by $(d)_\pi$  and $N!_\pi$, respectively.

For two sequences of complex numbers $(a_n)_{n\geq 1}$ and $(b_n)_{n\geq 1}$ such that the formal series expansions hold
$$\log\big(1+\sum_{i=1}^\infty \frac{a_n}{n!} z^n\big)=\sum_{i=1}^\infty \frac{b_n}{n!} z^n, $$
we have the following formula which expresses $a_n$  in terms of $b_1,\dots,b_n$ as a sum over partitions
\begin{equation} \label{reccla}
a_n = \sum _{\pi \in \mathcal{P}(n)}  b_\pi \qquad \text{for }  n\in \mathbb{N}.
\end{equation}
This last formula is equivalent to the  recursion formula
\begin{equation}\label{reccla2}  a_k = \sum_{i=1}^k \binom{k-1}{i-1} b_i a_{k-i} \qquad \text{for } k\in \mathbb{N}. \end{equation}
Conversely, using M\"obius inversion formula we can write $b_n$ in terms of $a_1,\dots, a_n$, \begin{equation*}
 b_n = \sum _{\pi \in \mathcal{P}(n)} a_{\pi} (-1)^{\vert \pi \vert -1} (\vert \pi \vert -1)! \qquad \text{for } n \in \mathbb{N}.
\end{equation*}

\subsection{Free Convolution}

Free convolution was defined in \cite{Voi85} for compactly supported probability measures.
Let $G_\mu(z)$ be the Cauchy transform of $\mu \in \mathcal{M}$ defined by
$$G_\mu(z)=\int \frac{1}{z-t} \mu(dt).$$

For $z$ near infinity, $G_\mu$ has a compositional inverse, $G_\mu^{-1}$. We define the $R$-transfrom by
$$R_\mu(z)=G_\mu^{-1}(z)-\frac{1}{z}.$$

The \emph{free additive convolution} of two probability measures $\mu_1,\mu_2$ on $\mathbb{R}$ is the
probability measure $\mu_1\boxplus\mu_2$ on $\mathbb{R}$ such that 
$$R_{\mu_1\boxplus\mu_2}(z) = R_{\mu_1}(z) + R_{\mu_2}(z).$$ Free additive convolution corresponds to the sum of free random variables: $\mu_a\boxplus\mu_b=\mu_{a+b}$, for $a$ and $b$ free random variables. 

The \emph{free cumulants} \cite{Sp94} are the coefficients
$r_n= r_n (\mu)$ in the series expansion 
\begin{equation}\label{free cumulants}
R_\mu(z) = \sum_{n=1}^\infty r_n z^{n}.
\end{equation}

The sequence $(r_n)_{n\geq1}$ satisfies the moment-cumulant formula
\begin{equation}\label{MCF}
m_n =\sum_{\pi\in \mathcal{NC}(n)}r_{\pi},
\end{equation}
where $m_n$ denotes the $n$-th moment of the measure $\mu$.

\section{Finite Free Cumulants}

In this section  we want to define a notion of cumulants which satisfy desired properties so that we can calculate finite free convolution using these cumulants.

\subsection{Defining cumulants} Our viewpoint comes from the following axiomatization of cumulants of Lehner \cite{Lehner}.
 \begin{defi} Given a notion of independence on a non commutative probability space $(A, \phi)$, a sequence of maps $a\to k_n(a)$, $n= 1, 2, . . .$ is called a cumulant sequence if it
satisfies
\begin{enumerate}
\item $k_n(a)$ is a polynomial in the first $n$ moments of $a$ with leading term $m_n(a)$. 
\item Homogeneity: $k_n(\lambda a) = \lambda^n k_n(a)$
\item Additivity: if $a$ and $b$ are “independent” random variables, then $k_n(a + b) = k_n(a) +
k_n(b).$
\end{enumerate}
\end{defi}

We will define cumulants  for finite free additive convolution that we call \emph{finite free cumulants} which satisfy similar properties in the set of monic polynomials of degree $d$. 

First, Property (1) ensures that the moments can be recovered from the cumulants and vice versa. We will replace it by
\begin{enumerate}
\item[(1')] $k_n(p)$ is a polynomial in the first $n$ moments of $a$ with leading term $\frac{d^n}{(d)_n} m_n(a)$. 
\end{enumerate}
While it may seem strange that the leading term is $\frac{d^n}{(d)_n} m_n(a)$, in practice, there is not much difference (see Remark \ref{variation}) and it will be more convenient to stay with this convention since it will be consistent with the definition of the $R$ transform of Marcus \cite{Mar}.

Homogeneity makes cumulants behave well with respect to \emph{dilations}. To rephrase homogeneity for polynomials we need to translate what dilation means.  We define the dilation by $\lambda\neq 0$ of the polynomial $p$ (of degree d) by the formula 
$$D_\lambda p(x):=\lambda^{-d} p(\lambda x),$$
and $D_0(p(x))=x^d$. Thus we replace homogeneity by 
\begin{enumerate}
\item[(2')] for all monic polynomial $p(x)$ we have $k_n(D_\lambda p(x)) = \lambda^n k_n(p(x)).$
\end{enumerate}

Finally, additivity which is possibly the most important property of cumulants, ensures that they behave well with respect to convolution and in this case of polynomial convolution will be replaced by
\begin{enumerate}
\item[(3')]for all monic polynomials $p$ and $q$, we have  $k_n(p\boxplus_d q) = k_n(p)+k_n(q).$
\end{enumerate}

Our starting point for defining finite free cumulants is the following formula for the finite $R$-transform which follows from Corollary 4.5 in \cite{Mar}.
\begin{equation} \label{definition R}
\mathcal{R} ^d _{p} (s) \equiv - \frac{1}{d} \frac{\partial}{\partial s} \ln \left( \sum_{i=0}^d  \frac{(-d)^i  a_i}{(d)_i} s^{i} \right) \qquad \text{mod } [s^{d}],
\end{equation} when $p$ is  the monic polynomial $p(x) = \sum_{i=0}^d x^{d-i} (-1)^i  a^p_i$.

Since one can recover a polynomial from the first $d$ coefficients of the finite $R$-transform it makes sense to consider a truncated $R$-transform given by
the sum of the first $d$ terms in the series expansion of $\mathcal{R} ^d _{p} $.  This truncated $R$-transform will have the cumulants as coefficients.

\begin{defi}
\label{defcumfin}
Let $p$ be a monic polinomial of degree d, and  suppose the $\mathcal{R}^d_{p} (s)$ satisfies  
$$\mathcal{R}^d_{p} (s) \equiv \sum_{j=0}^{d-1} \kappa^p_{j+1} s^j \quad \text{ mod } [ s^d ].$$

\begin{enumerate}
\item We call the sum of the first $d$ terms in the series expansion of $\mathcal{R} ^d$ the \emph{truncated $R$-transform} and denote by $\mathcal{\tilde{ R}}^d_{p} (s)$, i.e. $$\mathcal{\tilde{ R}}^d_{p} (s) :=\sum_{j=0}^{d-1} \kappa^p_{j+1} s^j. $$

\item The numbers $\kappa^p_1, \kappa^p_2, \dots , \kappa^p_{d}$ will be called the  ($d$-)finite free cumulants.
\end{enumerate}
\end{defi} 

Our goal will be to recover the truncated $R$-transform, and thus the finite free cumulants in terms of the polynomial $p$.
Below we will give precise combinatorial formulas between the coefficients, the finite free cumulants and the moments of a polynomial. We will use this formulas to show that these cumulants satisfy the desired properties.

\begin{prop}
Given $d\in\mathbb{N}$  the finite free cumulants defined by $p\mapsto \kappa_n^p$, $n= 1, 2, \dots, d$ are the unique maps satisfying  properties (1'), (2') and (3').
\end{prop}

It is standard that there are unique maps satisfying properties (1'), (2') and (3').

\subsection{Finite free cumulants are additive}

We start by relating the coefficients of a polynomial $p$ of degree $d$ and its finite free cumulants.

\begin{prop}[Coefficient-cumulant formula] \label{aaacum}
Let $p(x) = \sum_{i=0}^d x^{d-i} (-1)^i  a^p_i$ be a polynomial of degree $d$ and let $(\kappa^p_n)^d_{n=1}$ be its finite free cumulants. The following formulas hold.
\begin{equation} \label{aaatercum}
a^p_{n} = \frac{(d)_n}{d^nn!} \sum _{\pi \in \mathcal{P}(n)} d ^{| \pi |}  \mu(0_n,\pi) \kappa^p_{\pi}, \qquad n\in\mathbb{N}.
\end{equation}

\begin{equation} \label{cumteraaa}
\kappa^p_{n} = \frac{(-d)^n}{d(n-1)!} \sum _{\pi \in \mathcal{P}(n)} \frac{(-1)^{| \pi |} N!_{\pi} a^p_{\pi} (| \pi | -1)!}{(d)_{\pi} } \qquad n=1, \ldots, d,
\end{equation}
where $N!_{\pi}:= \prod_{V \in \pi} (| V |)!$.
\end{prop}
\begin{proof}

Let us define for $i=1,...,d$ the number $\tilde a_i$ and $b_i$ by the formulas 
$$\tilde a_i = \frac{(-d)^i  a^p_i}{(d)_i} i!  \qquad b_i = -d \kappa^p_{i} (i-1)! . $$
By definition
\begin{equation*}
\sum_{j=1}^{d} \kappa^{p}_{j} s^{j-1} \equiv \mathcal{R} ^d _{p} (s) \equiv - \frac{1}{d} \frac{\partial}{\partial s} \ln \left( \sum_{i=0}^d  \frac{(-1)^i  a^p_i d^{i}}{(d)_i} s^{i} \right) \qquad \text{mod } [s^{d}].
\end{equation*}
Multiplying by $-d$ and substituting we get
$$\frac{\partial}{\partial s} \ln \left( \sum_{i=0}^{d}  \frac{\tilde a_i}{i!} s^i \right) \equiv \sum_{j=1}^{d} \frac{b_{j}}{(j-1)!} s^{j-1} \quad \text{mod } [s^{d}].$$
Since $\tilde a_0 = 1$ by Eq (\ref{reccla2}) this is equivalent to 
\begin{equation}\label{aux3}
\tilde a_n = \sum_{\pi \in \mathcal{P}(n)} b_{\pi} \hspace{1cm} \text{for } n = 1, \ldots d.
\end{equation}
Now notice that by formula (\ref {Moeb1}), for $\pi=\{V_1,\dots,V_n\}$ we have 
$$b_{\pi}=b_{|V1|}\cdots b_{|V_n|}=(-d)^{|\pi|}(|V_1|-1!)\cdots (|V_n|-1)!k^p_\pi =(-1)^n d^{|\pi|}\mu(0_n,\pi) \kappa^p_\pi .$$
and thus (\ref{aux3}) is equivalent to (\ref{aaatercum}).
By applying M\"obius inversion in the set of partitions $\mathcal{P}(n)$, we obtain 
\begin{equation*}
b_n = \sum _{\pi \in \mathcal{P}(n)} \tilde a_{\pi} (-1)^{| \pi | -1} (| \pi | -1)! \qquad \text{for } n = 1, \ldots d,
\end{equation*}
from which, similarly, follows the second formula by substituting $\tilde a_n$ and $b_n$.
\end{proof}

\begin{rem} \begin{enumerate}
\item Formula (\ref{cumteraaa}) could be taken as definition of cumulants. 
\item Notice that formula (\ref{aaatercum}) makes sense and is valid, trivially, also for $n>d$ since the left hand-side vanishes.
\item Is is noteworthy that (\ref{cumteraaa}) only depends on the first $n$-cumulants, since there is a priori no reason for this. 
\end{enumerate}
\end{rem}

Now we show that finite free cumulants are additive with respect to polynomial convolution.

\begin{prop}
Let $p$ and $q$ be monic polynomials of degree $d$. Then
\begin{equation*}
\kappa_k^{p \boxplus_d q } = \kappa_k^p + \kappa_k^q \qquad  k = 1, \ldots, d.
\end{equation*}
\end{prop}

\begin{proof}
Recall that 
\begin{equation*}
\frac{a^{p \boxplus_d q }_k }{(d)_k} = \sum_{i+j=k} \frac{a^p_i}{(d)_i} \frac{a^q_j}{(d)_j}.
\end{equation*}
substituting $\frac{a^p_i}{(d)_i}$ and $\frac{a^q_j}{(d)_j}$ by the formulas in \eqref{aaatercum} we get that the right-hand side of the equation above is given by
\begin{eqnarray}
\sum_{i+j=k} \left( \frac{1}{d^i i!} \sum _{\pi \in \mathcal{P}(i)} d^{| \pi |} \mu(0,\pi) \kappa^p_{\pi} \right) \left( \frac{1}{d^j j!} \sum _{\pi \in \mathcal{P}(j)} d^{| \pi |} \mu(0,\pi) \kappa^q_{\pi} \right) \nonumber \\
= \frac{1}{d^k k!} \sum_{i+j=k} \binom{k}{i} \left( \sum _{\pi \in \mathcal{P}(i)} d^{| \pi |} \mu(0,\pi) \kappa^p_{\pi} \right) \left( \sum _{\pi \in \mathcal{P}(j)} d^{| \pi |} \mu(0,\pi) \kappa^q_{\pi}  \right). \nonumber
\end{eqnarray}
Since, if $i+j=k$, we can interpret the binomial coefficient, $\binom{k}{i}$, as the number of ways of dividing the set $[k]$ in two subsets,  the first one having size $i$ and the second one having size $j$, we may rewrite the last sum as follows,
\begin{eqnarray}
&&\sum_{i+j=k} \binom{k}{i} \left( \sum _{\pi \in \mathcal{P}(i)}d^{| \pi |} \mu(0,\pi) \kappa^p_{\pi} \right) \left( \sum _{\pi \in \mathcal{P}(j)} d^{| \pi |} \mu(0,\pi) \kappa^q_{\pi}  \right) \nonumber \\
&=& \sum_{S \subset [k]} \left( \sum _{\pi_1 \in \mathcal{P}(S)} d^{| \pi_1 |} \mu(0,\pi_1) \kappa^p_{\pi_1} \right) \left( \sum _{\pi_2 \in \mathcal{P}(S^c)} d^{| \pi_2 |} \mu(0,\pi_2) \kappa^q_{\pi_2}  \right) \nonumber \\
&=& \sum_{\pi_1 \cup \pi_2 \in \mathcal{P}(k)} \left( d^{| \pi_1 |} \mu(0,\pi_1) \kappa^p_{\pi_1} \right) \left( d^{| \pi_2 |} \mu(0,\pi_2) \kappa^q_{\pi_2}  \right) \nonumber.\\
& = & \sum _{\pi \in \mathcal{P}(k)} \sum_{\pi_1 \cup \pi_2 = \pi} \left( d^{| \pi_1 |} \mu(0,\pi_1) \kappa^p_{\pi_1} \right) \left(d^{| \pi_2 |} \mu(0,\pi_2) \kappa^q_{\pi_2}  \right) \nonumber \\
& = & \sum _{\pi \in \mathcal{P}(k)} \sum_{\pi_1 \cup \pi_2 = \pi} d^{| \pi |} \mu(0,\pi) \kappa^p_{\pi_1} \kappa^q_{\pi_2} \nonumber \\
& = & \sum _{\pi \in \mathcal{P}(k)} d^{| \pi |} \mu(0,\pi) \sum_{\pi_1 \cup \pi_2 = \pi} \kappa^p_{\pi_1} \kappa^q_{\pi_2}  \nonumber \\
& = & \sum _{\pi \in \mathcal{P}(k)} d^{| \pi |} \mu(0,\pi) (\kappa^p  + \kappa^q )_{\pi}, \nonumber
\end{eqnarray}
where we used that the M\"obius function is multiplicative.
Thus, we obtain that 
\begin{equation*}
a^{p \boxplus_d q}_{k} = \frac{(d)_k}{d^kk!} \sum _{\pi \in \mathcal{P}(k)}  (\kappa^p  + \kappa^q )_{\pi}    d^{| \pi |} \mu(0,\pi) \qquad  k=1, \ldots, d.
\end{equation*}
Finally, since Equation \eqref{aaatercum} applied to $p \boxplus_d q$ is given by
\begin{equation*}
a^{p \boxplus_d q}_{k} = \frac{(d)_k}{d^kk!} \sum _{\pi \in \mathcal{P}(k)}  \kappa^{p \boxplus_d q}_{\pi}  d^{| \pi |} \mu(0,\pi) \qquad k=1, \ldots, d.
\end{equation*}
we get the conclusion.
\end{proof}

\begin{cor}The truncated  $R$-transform is additive w.r.t to finite free convolution,  i.e. $$\mathcal{\tilde{ R}}^d_{p\boxplus_d q}=\mathcal{\tilde{ R}}^d_{p}+\mathcal{\tilde{ R}}^d_{q}.$$

\end{cor}

\section{Moment-cumulant formulas}

For a a polynomial $p$ of degree $d$ with $p(x)=(x-r_1)(x-r_2)\dots (x-r_d)$ we define the $n^{\text{th}}$ moment of $p$ by 
$$m_n^p:=\frac{1}{d}\sum^d_{i=1} r_i^n.$$ 
In this section we will give a moment-cumulant formula. This is one of the main contributions of this paper. Before doing this we state the relation between the moments and the coefficients of $a_n$ which follows from Newton Identities. 

\begin{lem}[Moment-coefficient formula]
Let $p(x) = \sum_{i=0}^d x^{d-i} (-1)^i  a^p_i$ be a polynoimial of degree $d$ and let  $(m_n)^\infty_{n=1}$ be the moments of $p$.  Then
\begin{equation}
a^p_n = \frac{1}{ n!} \sum _{\pi \in \mathcal{P}(n)} d^{| \pi |} \mu(0,\pi)  m^p_{\pi}  \qquad  n=1, \ldots, d,
\label{aaatermom}
\end{equation}
and
\begin{equation}
m^p_n = \frac{(-1)^n}{d(n-1)!} \sum _{\pi \in \mathcal{P}(n)} (-1)^{| \pi |}  N!_{\pi}(\vert \pi \vert -1)!  a^p_{\pi} \qquad  n\in\mathbb{N}.
\label{momteraaa}
\end{equation}
\end{lem}

\begin{proof}
Newton identities state that the sequences  $b_k = r_1^k + r_2^k + \ldots + r_d^k$ and the coefficients $a_k= \sum_{sym} r_1r_2 \cdots r_k$ satisfy the relation
\begin{equation}\label{NI}
ka^p_k = \sum_{i=1}^k (-1)^{i-1} a^p_{k-i} b_i \qquad  k=1,\ldots, d.
\end{equation}
Now let us define the quantities 
$$\tilde a_k = k!a^p_k, \qquad \tilde b_k = (-1)^k(k-1)! b_k \qquad k=1,\ldots, d.$$
Then (\ref{NI}) is equivalent to 
$$\tilde a_k = \sum_{i=1}^k \binom{k-1}{i-1} \tilde b_i \tilde a_{k-i} \qquad \text{for } k = 1, \ldots, d.$$
By Eqs.  (\ref{reccla}) and (\ref{reccla2}) we have that
\begin{equation*}
\tilde a_n = \sum _{\pi \in \mathcal{P}(n)} \tilde b_\pi \qquad \text{for } n=1, \ldots, d.
\end{equation*}
Finally, by M\"obius inversion formula we obtain
\begin{equation*}
\tilde b_n = \sum _{\pi \in \mathcal{P}(n)} \tilde a_{\pi} (-1)^{\vert \pi \vert -1} (\vert \pi \vert -1)! \qquad \text{for } n = 1, \ldots d.
\end{equation*}
Substituting the values of $\tilde a_i$ and $\tilde b_i$ and observing that $\tilde b_i=d m_i$, we obtain the desired formulas.
\end{proof}

Now we are able to prove the moment-cumulant formulas. We will omit the superscript $p$ to avoid heavy notation.
\begin{thm}
Let $p$ be a polynomial of degree $d$ and let  $(m_n)^\infty_{n=1}$ and $(\kappa_n)^d_{n=1}$, be the moments and cumulants of $p$, respectively. Then
\begin{eqnarray} \label{cum-mom formula}
\kappa_n &=& \frac{(-d)^{n-1}}{(n-1)!} \sum_{\sigma \in \mathcal{P}(n)} d^{|\sigma|}\mu(0,\sigma) m_{\sigma}  \sum_{\pi \geq \sigma} \frac{\mu(\pi,1_n)}{(d)_{\pi}},
\\&=& \frac{(-1)^n d^{n-1}}{(n-1)!} \sum_{\sigma \in \mathcal{P}(n)} d^{|\sigma|}\mu(0,\sigma)m_{\sigma} \sum_{\pi \geq \sigma} \frac{(-1)^{\vert \pi \vert} (\vert \pi \vert -1 )!}{(d)_{\pi}}, \nonumber
\end{eqnarray}
for $n=1,\ldots, d$ and
\begin{eqnarray}\label{mom-cum formula}
m_n   &= &\frac{(-1)^n}{d^{n+1}(n-1)!} \sum_{\sigma \in \mathcal{P}(n)}d^{|\sigma|}\mu(0,\sigma)\kappa_\sigma \sum_{\pi \geq \sigma}- \mu(\pi,1_n)  (d)_\pi \\
 &=& \frac{(-1)^n}{d^{n+1}(n-1)!} \sum_{\sigma \in \mathcal{P}(n)}d^{|\sigma|}\mu(0,\sigma)\kappa_\sigma \sum_{\pi \geq \sigma} (-1)^{\vert \pi \vert} (d)_\pi(\vert \pi \vert -1)!
 \nonumber
\end{eqnarray}
for $n\in\mathbb{N}$.
\end{thm}

\begin{proof}
Formula  \eqref{cumteraaa} tells us that
\begin{eqnarray}
\kappa_n &=&  \frac{(-d)^n}{d(n-1)!} \sum _{\pi \in \mathcal{P}(n)}  \frac{(-1)^{\vert \pi \vert} N!_{\pi} a_{\pi} (\vert \pi \vert -1)!}{(d)_{\pi} } \nonumber \\ &=& \frac{(-d)^n}{d(n-1)!} \sum _{\pi \in \mathcal{P}(n)} \frac{(-1)^{\vert \pi \vert} (\vert \pi \vert -1 )!}{(d)_{\pi}} \prod_{i=1}^j \vert V_i \vert! a_{\vert V_i \vert}, \nonumber
\end{eqnarray}
where  $\pi = \{ V_1, \ldots, V_j \}$. The last product can be rewritten in terms of the moments, by using formula \eqref{aaatermom}:
\begin{eqnarray}
\prod_{i=1}^j \vert V_i \vert! a_{\vert V_i \vert} &=& \prod_{i=1}^j  \sum_{\sigma_i \in \mathcal{P}(\vert V_i \vert)} d^{|\sigma_i|}\mu(0,\sigma_i)m_{\sigma_i}  \nonumber \\
&=&  \prod_{i=1}^j \sum_{\sigma_i \in \mathcal{P}(\vert V_i \vert)} d^{|\sigma_i|}\mu(0,\sigma_i) m_{\sigma_i}  \nonumber \\
&=&  \sum_{\substack{ \sigma = \sigma_1 + \ldots + \sigma_j \\ \sigma_1 \in \mathcal{P}(\vert V_1 \vert), \ldots, \sigma_j \in \mathcal{P}(\vert V_j \vert) }} d^{|\sigma|}\mu(0,\sigma) m_{\sigma} \nonumber \\
&=& \sum_{\sigma \leq \pi} d^{|\sigma|}\mu(0,\sigma) m_{\sigma} . \nonumber
\end{eqnarray}
So we conclude that
\begin{eqnarray}
\kappa_n &=& \frac{(-d)^n}{d(n-1)!} \sum _{\pi \in \mathcal{P}(n)} \frac{(-1)^{\vert \pi \vert} (\vert \pi \vert -1 )!}{(d)_{\pi}} \sum_{\sigma \leq \pi} d^{|\sigma|}\mu(0,\sigma) m_{\sigma}. \nonumber 
\end{eqnarray}
Changing the order of summation, by first considering the index $\sigma$ and then the index $\pi$ we can rewrite the previous formula as
\begin{equation*}
\kappa_n = \frac{(-1)^n d^{n-1}}{(n-1)!} \sum_{\sigma \in \mathcal{P}(n)} d^{|\sigma|}\mu(0,\sigma)m_{\sigma} \sum_{\pi \geq \sigma} \frac{(-1)^{\vert \pi \vert} (\vert \pi \vert -1 )!}{(d)_{\pi}},
\end{equation*}
for $n=1,\ldots, d$.

The proof of  (\ref{mom-cum formula})  follows the same lines, using instead formulas \eqref{momteraaa} and \eqref{aaatercum}.
\end{proof}

\begin{cor}
Finite free cumulants satisfy properties (1') and (2'). 
\end{cor}
\begin{proof}Property (1') follows since for $\sigma=1_n$ the quantity $$(-1)^n \frac{d^{n-1}}{(n-1)!} d^{|\sigma|}\mu(0,\sigma) \sum_{\pi \geq \sigma} \frac{(-1)^{\vert \pi \vert}(\vert \pi \vert -1)!}{ (d)_\pi}$$ equals
$$\frac{d^{n-1}}{(n-1)!} (-d) (n-1)! \frac{-1}{ (d)_{n}}=\frac{d^{n}}{(d)_n}.$$
Homogeneity follows from formula $(\ref{cum-mom formula})$ since $m_\sigma$ also satisfies it.
\end{proof}

\begin{rem} \label{variation}  If one insists that cumulants to have a leading term $m_n$ the formulas above are very similar.  Indeed, we could have defined cumulants to be $\tilde \kappa_n=\frac{(d)_n}{d^n}\kappa_n.$ In this case formulas (\ref{cum-mom formula}) and (\ref{mom-cum formula})
may be rewritten as
\begin{equation} \label{mum-mom formula 2}
\tilde\kappa_n = \frac{(-1)^n(d)_{n-1}}{(n-1)!} \sum_{\sigma \in \mathcal{P}(n)} d^{|\sigma|}\mu(0,\sigma) m_{\sigma}\sum_{\pi \geq \sigma} \frac{(-1)^{\vert \pi \vert} (\vert \pi \vert -1 )!}{(d)_{\pi}},
\end{equation}
for $n=1,\ldots, d$ and
\begin{equation}\label{mom-cum formula 2}
m_n = \frac{(-1)^n}{d^{n+1}(n-1)!} \sum_{\sigma \in \mathcal{P}(n)}(d)_\sigma \mu(0,\sigma) \tilde \kappa_\sigma \sum_{\pi \geq \sigma} (-1)^{\vert \pi \vert} (d)_\pi(\vert \pi \vert -1)!.
\end{equation}
\end{rem}

In Section \ref{aproach} we will analyze the finite free cumulants in the asymptotic regime. Before doing this, we will give a more combinatorial formula for the polynomial \[P_\sigma (d) := \sum_{\pi \geq \sigma} (-1)^{| \pi |} (d)_\pi(| \pi | -1)!, \] 
appearing in Equation (\ref{mom-cum formula}). We thank Drew Armstrong for pointing this lemma.

\begin{lem} For all $\sigma\in \mathcal{P}(n)$ the following formula holds
\[
P_\sigma (d) =\sum_{\rho \lor \sigma = 1_n} d^{| \rho |}  \mu (0, \rho).
\]
\end{lem}

\begin{proof}
By the known equation $(d)_k = \sum_{\pi \in \mathcal{P}(k)} \mu (0, \pi) d^{| \pi |}$, we can rewrite the polynomial $P_\sigma(d)$ as

\begin{eqnarray*}
P_\sigma (d) &=& \sum_{\pi \geq \sigma} (-1)^{| \pi |} (d)_\pi(| \pi | -1)!
=\sum_{\pi \geq \sigma} \mu (\pi, 1_n) \sum_{\rho \leq \pi}\mu (0,\rho)d^{| \rho |}\\
&=& \sum_{\rho\in \mathcal{P}(n) }\mu (0,\rho)d^{| \rho |} \sum_{\pi \geq \sigma,\pi\geq \rho} \mu (\pi, 1_n) 
= \sum_{\rho \in P(n) }\mu (0,\rho)d^{| \rho |} \sum_{\pi \geq \sigma\vee\rho} \mu (\pi, 1_n)\\
&=&\sum_{\rho:\sigma\vee\rho =1_n }\mu (0,\rho)d^{| \rho |},
\end{eqnarray*}
where we used in the last inequality the fact that when $\sigma\vee\rho \neq 1_n.$ $$\sum_{\pi \geq \sigma\vee\rho} \mu (\pi, 1_n)=0.$$
\end{proof}

As a corollary we may rewrite the moment-cumulant formula.
\begin{cor} \label{m-c 3}
\begin{eqnarray}\label{m-c 2} m_n&=& \frac{(-1)^n}{d^{n+1}(n-1)!} \sum_{\sigma \in \mathcal{P}(n)} d^{|\sigma|}\mu(0,\sigma) \kappa_\sigma \sum_{\rho:\rho \lor \sigma = 1_n} d^{| \rho |}  \mu (0, \rho)
\\ &=& \frac{(-1)^n}{d^{n+1}(n-1)!} \sum_{\substack{\sigma,\rho \in \mathcal{P}(n)\\ \rho \lor \sigma = 1_n}} d^{| \rho |+|\sigma|}  \mu (0, \rho) \mu(0,\sigma) \kappa_\sigma 
\end{eqnarray}

\end{cor}

\section{Cumulants approximate free cumulants}

In order to analyze formula (\ref{m-c 2}) in the asymptotic regime we need to analyze the asymptotic behavior of the polynomial 
$P_\sigma (d)$. This is the content of the next proposition. 

\begin{prop}
For all  $\sigma \in \mathcal{P}(n)$ the polynomial  $P_\sigma (d)$ is a polynomial  (in d) of degree $n+1 - | \sigma |$ with leading coefficient $\frac{(-1)^{| \sigma |} (n-1)! n_\sigma}{(n+1 - | \sigma |)!}$,  where for $\sigma=\{V_0,\dots, V_{m-1}\}$, $n_\sigma$ is given by $n_\sigma=|V_0||V_2|\cdots|V_{m-1}|$.
\end{prop}

\begin{proof}

Denote the blocks of  $\sigma$ by $V_0, V_1, \ldots, V_{m-1}$ with $n \in V_0$, and let us define $\sigma_1, \ldots, \sigma_{m-1} \in \mathcal{P}(n)$ by $\sigma_i = \{ V_1, \ldots, V_{i-1}, V_i \cup V_{0} , V_{i+1}, \ldots,  V_{m-1} \}$ for $i=1, \ldots, m-1$.

Now, for a given a partition $\pi = \{W_1, \ldots, W_r \} \in \mathcal{P}(n-1)$ and $k \in [n-1]$, we will denote $i_k$ the index of the block in $\pi$ containing the element $k$ ($k \in W_{i_k}$) and $\pi_{i_k} = \{W_1, \ldots, W_{i_k} \cup \{n\}, \ldots, W_r \} \in \mathcal{P}(n)$. We will also consider the partition $\pi_{0} = \{W_1, \ldots, W_r, \{n\} \} \in \mathcal{P}(n)$. 

Now, let us define
$$\hat{\cdot} : \mathcal{P}(n) \to \mathcal{P}(n-1)$$
$$ \pi \mapsto \hat{\pi} = \{W_1, \ldots, W_{i_n} \backslash \{n \}, \ldots, W_{k} \},$$
where $\pi = \{W_1, \ldots, W_{k} \}$.

Observe that given $\hat{\rho} = \{ U_1, \ldots, U_r \} \in \mathcal{P}(n-1)$ then $\rho \in \{\hat{\rho}_0, \hat{\rho}_1, \ldots, \hat{\rho}_r \}$ and that from the properties of the M\"obius function we have that $d^{| \rho_0 |} \mu (0, \rho_0) = d^{| \hat{\rho} |+1} \mu (0, \hat{\rho})$ and $d^{| \rho_i |} \mu (0, \rho_i) = -d^{| \hat{\rho} |} | U^i | \mu (0, \hat{\rho})$ for $i=1, \ldots, r$.

On the other hand we have that
$$\rho_{i_k} \lor \sigma = (\rho_0 \lor (0_{n-1})_{i_k})\lor \sigma = \rho_0 \lor ((0_{n-1})_{i_k})\lor \sigma)= \rho_0 \lor \sigma_{i_k}.$$
Since the block in $\rho_0$ containing $n$ has only one element and the block in $\sigma_{i_k}$ containing $n$ has more than one element we get that
$$\rho_{i_k} \lor \sigma = 1_n \quad \Leftrightarrow \quad \rho_0 \lor \sigma_{i_k} = 1_n \quad \Leftrightarrow \quad \rho \lor \hat{\sigma_{i_k}} = 1_{n-1}.$$

Therefore we get  
\begin{eqnarray*}
&& \sum_{i=1}^r d^{| \rho_i |}  \mu (0, \rho_i) I_{\rho_i \lor \sigma = 1_n} \sum_{k=1}^{n-1} \frac{d^{| \rho_{i_k} |}  \mu (0, \rho_{i_k})}{| U_k |} I_{\rho_{i_k} \lor \sigma = 1_n} = \sum_{k=1}^{n-1} \frac{- d^{| \rho |} | U_k | \mu (0, \rho}{| U_k |} I_{\rho_{i_k} \lor \sigma = 1_n} \\
&& \quad = -\sum_{k=1}^{n-1} d^{| \rho |}  \mu (0, \rho) I_{\rho \lor \hat{\sigma_{i_k}} = 1_{n-1}} = -\sum_{j=0}^{m-1} | V_j | d^{| \rho |} \mu(0, \rho) I_{\rho \lor \hat{\sigma_{j}} = 1_{n-1}}.
\end{eqnarray*}

Now, we divide in two cases depending on whether $| V |> 1$ or $| V_{0} | =1.$

For the case when $| V_{0} | > 1$ ($\sigma \nleq (1_{n-1})_0$) we can conclude that
\begin{eqnarray*}
P_\sigma (d) &=& \sum_{\rho \in \mathcal{P}(n-1)} \sum_{i=0}^r d^{| \rho_i |}  \mu (0, \rho_i) I_{\rho_i \lor \sigma = 1_n} \\
&=& \sum_{\rho \in \mathcal{P}(n-1)} d^{| \rho_0 |}  \mu (0, \rho_0) I_{\rho_0 \lor \sigma = 1_n}  + \sum_{\rho \in \mathcal{P}(n-1)} \sum_{i=1}^r d^{| \rho_i |}  \mu (0, \rho_i) I_{\rho_i \lor \sigma = 1_n}  \\
&=& \sum_{\rho \in \mathcal{P}(n-1)} d^{| \rho |+1}  \mu (0, \rho) I_{\rho \lor \hat{\sigma} = 1_{n-1}} - \sum_{\rho \in \mathcal{P}(n-1)} \sum_{j=0}^{m-1} | V_j | d^{| \rho |}  \mu (0, \rho) I_{\rho \lor \hat{\sigma_{j}} = 1_{n-1}} \\
&=& d\sum_{\rho \lor \hat{\sigma} = 1_{n-1}} d^{| \rho |}  \mu (0, \rho) -\sum_{j=0}^{m-1} | V_j |  \sum_{\rho \in \mathcal{P}(n-1)}  d^{| \rho |}  \mu (0, \rho) I_{\rho \lor \hat{\sigma_{j}} = 1_{n-1}} \\
&=& d P_{\hat{\sigma}}(d) - \sum_{j = 0}^{m-1} | V_j | P_{\hat{\sigma_{j}}}(d).
\end{eqnarray*}

On the other hand, for the case $| V_{0} | = 1$ ($\sigma \leq (1_{n-1})_0$) the condition $\rho_0 \lor \sigma = 1_n$ is never fulfilled for any $\rho$ (because $\rho_0 \lor \sigma \leq (1_{n-1})_0 < 1_n$), and we can observe that $\hat{\sigma}= \hat{\sigma_{0}} = \hat{\sigma_{1}} = \cdots = \hat{\sigma_{m-1}}$ and therefore in this case we get
\begin{eqnarray*}
P_\sigma (d) &=& \sum_{\rho \in \mathcal{P}(n-1)} d^{| \rho_0 |}  \mu (0, \rho_0) I_{\rho_0 \lor \sigma = 1_n}  + \sum_{\rho \in \mathcal{P}(n-1)} \sum_{i=1}^r d^{| \rho_i |}  \mu (0, \rho_i) I_{\rho_i \lor \sigma = 1_n}  \\ 
&=& - \sum_{j = 0}^{m-1} | V_j | P_{\hat{\sigma_{j}}}(d)   = - P_{\hat{\sigma}}(d) \sum_{j = 0}^{m-1} | V_j |  \\
&=& -(n-1) P_{\hat{\sigma}}(d) \\
\end{eqnarray*}

Now, it is easy to prove the claim by induction on $n$. For the case $l=| V_{0} | > 1$ one uses the recursion $P_\sigma (d) = d P_{\hat{\sigma}}(d) - \sum_{j = 0}^{m-1} | V_j | P_{\hat{\sigma_{j}}}(d) $.  The leading coefficient of $d P_{\hat{\sigma}}(d)$ is given by
\begin{equation}\label{primera parte}
\frac{(-1)^{| \sigma' |} ((n-1)-1)! n_{\sigma'} }{((n-1)+1 - | \sigma' |)!} = \frac{(-1)^{|\sigma|} (n-2)! n_\sigma}{(n - |\sigma|)!l} \left( l-1 \right),
\end{equation}
and the leading coefficient of $ \sum_{j = 0}^{m-1} | V_j | P_{\hat{\sigma_{j}}}(d)$ is given by
\begin{equation*}
\sum_{j = 1}^{k-1} | V_j | \left( \frac{(-1)^{k-1} (n-2)! \left( n_\sigma \frac{l-1 + | V_j | }{l| V_j |} \right)}{(n+1 -|\sigma|)!} \right)= \frac{(-1)^{k-1} (n-2)! n_\sigma}{(n+1-|\sigma|)!l} \left( (l-1)(m-1) + (n-l) \right). 
\end{equation*}

Subtracting the RHS from (\ref{primera parte}) we obtain the claim for this case. Finally,  the case $| V_{0} | = 1$ follows directly from the the formula $P_\sigma (d) = -(n-1) P_{\hat{\sigma}}(d)$.

\end{proof}
Now we can prove that finite free cumulants approach free cumulants and thus finite free convolution approaches free convolution as observed by Marcus \cite{Mar}.
\begin{thm}
The finite free cumulants converge to the free cumulants. That is,
$$\lim _{d \to \infty} \kappa_\sigma^{(d)} = r_\sigma$$
for all $\sigma \in \mathcal{P}$.
\end{thm}

\begin{proof}
First we rewrite the formula between moments and finite free cumulants in terms of the polynomial $P_\sigma(d)$
\begin{equation}
m_n = \sum_{\sigma \in \mathcal{P}(n)} \frac{ \kappa_\sigma}{d^{n+1-\vert \sigma \vert}}\frac{(-1)^n \mu(0,\sigma)}{(n-1)!}   P_{\sigma}(d). 
\end{equation}
Denote by $Q_\sigma(d):= \frac{(n+1 - \vert \sigma \vert)!}{(-1)^{\vert \sigma \vert} (n-1)! n_\sigma} P_{\sigma}(d)$. By the previous lemma $Q_\sigma(d)$ is a monic polynomial  of degree $n+1 -|\sigma |$. Moreover,
\begin{eqnarray}
\frac{(-1)^n \mu(0,\sigma)}{(n-1)!} P_{\sigma}(d)
&=& \frac{N!_\sigma}{(n+1 - \vert \sigma \vert)!} Q_\sigma(d).
\end{eqnarray}
Now, 
\begin{eqnarray}
m_n &=&  \sum_{\sigma \in \mathcal{P}(n)} \frac{\kappa_\sigma}{d^{n+1-\vert \sigma \vert}} \frac{N!_\sigma}{(n+1 - \vert \sigma \vert)!} Q_\sigma(d) \nonumber \\
&=& \sum_{\sigma \in \mathcal{P}(n)} \frac{N!_\sigma}{(n+1 - \vert \sigma \vert)!} \frac{Q_\sigma(d)}{d^{n+1-\vert \sigma \vert}} \kappa_\sigma \nonumber \\
&=& \sum_{\sigma \in \mathcal{NC}(n)} \frac{Q_\sigma(d)}{d^{n+1-\vert \sigma \vert}} \kappa_\sigma,\nonumber 
\end{eqnarray}
where the last equality follows from Lemma \ref{type} for the number of partitions and non-crossing partition per type. Since $Q_\sigma(d)$ es monic of degree $n+1-|\sigma |$,  taking limits for $\kappa_n^{(d)}$, we conclude that
\begin{eqnarray*}
\lim _{d \to \infty} \kappa_n^{(d)} &=& \lim _{d \to \infty} \frac{d^n}{(d)_n} \left( m_n - \sum_{\substack{\sigma \in \mathcal{NC}(n) \\ \sigma \neq 1_n}} \frac{Q_\sigma(d)}{d^{n+1-\vert \sigma \vert}} \kappa_\sigma ^{(d)} \right) \\
&=& m_n - \sum_{\substack{\sigma \in \mathcal{NC}(n) \\ \sigma \neq 1_n}} \lim _{d \to \infty} \frac{Q_\sigma(d)}{d^{n+1-\vert \sigma \vert}} \kappa_\sigma ^{(d)}  \\
&=& m_n - \sum_{\substack{\sigma \in \mathcal{NC}(n) \\ \sigma \neq 1_n}} \lim _{d \to \infty} \kappa_\sigma^{(d)} \\
&=& m_n - \sum_{\substack{\sigma \in \mathcal{NC}(n) \\ \sigma \neq 1_n}} r_\sigma \\
&=& r_n,
\end{eqnarray*}
where we used induction in the fourth equality.  This proves the claim for $\kappa_n$, $n \in \mathbb{N}$. The multiplicativity of cumulants shows the result for all partitions.
\end{proof}

\label{aproach}

\begin{cor}Finite free convolution approaches free convolution.
\end{cor}

\section{Applications and examples}

\begin{exa}[Central Limit Theorem] Now that we have proved that cumulants satisfy the desired properties, the Central Limit Theorem can be easily derived by a standard argument using them.  Indeed, let $p$ be a polynomial with $\kappa_1(p)=0$, $\kappa_2(p)=1$  and denote by $p_n=p\boxplus_d \cdots \boxplus_d p$ the $n$-fold finite free convolution of $p$. Then $\kappa_1(D_{1/\sqrt{n}} (p_n))=0$, $\kappa_2(D_{1/\sqrt{n}} (p_n))=1$
and for $r>2$ and $$\kappa_r(D_{1/\sqrt{n}} (p_n))=\frac{n}{n^{r/2}}\kappa_r(p)=\frac{1}{n^{r/2-1}} \kappa_r(p)\to 0.$$
From Proposition \ref{aaacum} we see that the coefficients of the limiting polynomial are given by $a_{2i+1}=0$ and $$a_{2i} = \frac{(d)_{2i}}{d^{2i}(2i)!} \sum_{\pi\in P_2(2i)} d^{i} (-1)^i=  \frac{(d)_{2i} }{d^{i}}\frac{ (-1)^i}{ 2^i i!},$$ where we used that the set $P_2(2i)$ of pair partitions of $[2i]$ has cardinality $\frac{(2i)!}{2^i i!}$.  Hence, we recover the result of Marcus \cite{Mar} that the Central Limit Theorem is given by the polynomial $d^{-d/2} H_d(\sqrt{d}x)$ \footnote{ Marcus considers the scaling $\kappa_2(p)=1-1/d$.} since  
$$H_{d}(x)=d!\sum _{i=0}^{\lfloor {\tfrac {d}{2}}\rfloor }{\frac {(-1)^{i}}{i!(d-2i)!}}{\frac {x^{d-2i}}{2^{i}}}.$$
\end{exa}

\begin{exa}[Poisson distribution]  A Poisson distribution is a distribution with cumulants $\kappa_n=\lambda$ for all $n$.  In the case of polynomials we assume that $\lambda d $ is a positive integer to obtain a valid polynomial.  So let $Poiss(\lambda,d)$ be the polynomial with finite free cumulants equal to $\lambda$ and degree $d$. For this case the coefficients are again very explicit in terms of Pochhamer symbols, 
$$a_n=\frac{(d)_n}{d^n n!}\sum_{\pi\in \mathcal{P}(n)} d^{|\pi|}\lambda^{|\pi|} \mu(0,\pi)=\frac{(d)_n}{d^n n!}\chi_{\mathcal{P}(n)}(d\lambda)=\frac{(d)_n}{d^n n!}(d\lambda)_n.$$
For $\lambda=1/d$ we obtain the $a_1=1$ and $a_n=0$ for $n\geq2$. i.e. $p(x)=x^d-x^{d-1}$ as observed by Marcus.  Notice that in this case something somewhat counter-intuitive happens: all cumulants and all moments equal $1/d$. 

In general for $\lambda d$ a positive integer we obtain a modification of the Laguerre polynomials.  The observation that $(d\lambda)_n=0$ if and only if $n\geq d\lambda+1$ shows that  for $\lambda<1$, $Poiss(\lambda,d)$ is of the form  $x^d+\cdots+c_{d\lambda} x^{d-d\lambda}$ and thus has a root at $0$ of multiplicity $d-d\lambda$ or, equivalently, has an atom of size $1-\lambda$.  This is in accordance with the usual free Poisson distribution. However, contrary to the free case, $Poiss(\lambda,d)$ is \textbf{never} infinitely divisible as it will follow from Proposition  \ref{divinf}.
\end{exa}

For a monic polynomial $p$ we denote by $p^{\boxplus_d t}$ the polynomial such that $t\kappa_n (p^{\boxplus_d t})=t \kappa_n (p)$. A monic polynomial with real roots is said to be infinitely divisible is for all $t>0$, the polynomial $p^{\boxplus_d t}$ has real roots.

\begin{prop}\label{divinf}
The following are necessary conditions for infinite divisibility.\begin{enumerate}
\item The sequence $\{\tilde \kappa_n\}^d_{n=1}$ is conditionally positive definite.
\item The sequence $\{\kappa_n\}^d_{n=1}$ is conditionally positive definite.
\end{enumerate}
\end{prop}

\begin{proof}
Suppose that $p$ is infinitely divisible and let $p_N$ be such that $p$ is the $N$-fold finite free convolution of $p_N$. i.e.  $p=p_N\boxplus_d \cdots \boxplus_d p_N$, then $N\tilde\kappa_n (p_N)=\tilde\kappa_n (p)$. Now, by the moment-cumulant formula we may rewrite $\tilde \kappa_n(p)$ as 

$$ N \tilde \kappa_n (p_N) = Nm_n(p_N) - N\sum_{\substack{\sigma \in \mathcal{NC}(n) \\ \sigma \neq 1_n}} \frac{Q_\sigma(d)}{d^{n+1-\vert \sigma \vert}}  N^{-|\sigma|}\tilde \kappa_\sigma(p) $$ for some fixed $Q_\sigma$. When $N$ approaches infinity we have that
$$\tilde \kappa(p)=\lim _{N \to \infty} N \tilde \kappa_n (p_N) = \lim _{N \to \infty} Nm_n(p_N) .$$
Consider now $\alpha_1, \ldots, \alpha_r  \in \mathbb{C}$. Then we have 
\begin{eqnarray*}
\sum_{i,j=1}^r \alpha_i \bar\alpha_j \tilde \kappa_{i+j}(p) &=& \lim_{N \to \infty} N \sum_{i,j=1}^r \alpha_i \bar\alpha_j m_{i+j}  \\
&=&  \lim_{N \to \infty} N \sum_{i,j=1}^r \alpha_i \bar\alpha_j \int x^{j+i} d\rho(x) \\
&=&  \lim_{N \to \infty} N \int \sum_{i,j=1}^r \alpha_i x^i \bar\alpha_j x^j d\rho(x) \\
&=&  \lim_{N \to \infty} N \int \left(\sum_{i=1}^r \alpha_j x^i \right)  \overline{\left(\sum_{j=1}^r\alpha_j x^j\right)} d\rho (x)\geq0,
\end{eqnarray*}
where $\rho$ is the spectral measure of $p$. This proves (1).

(1) implies (2). Indeed since  $\frac{d^n}{(d)_n}$ and $\tilde \kappa_n$ are conditionally positive definite then also $\kappa_n=\frac{d^n}{(d)_n}\tilde \kappa_n$.
\end{proof}

\begin{rem} The simplest example where (2) does not imply (1) is the $Poisson(1,d)$. That $Poisson(1,d)$ does not satisfy (1) can be seen from the fact that $\frac{(d)_n}{d^n}$ is not conditionally positive definite.
While (1) is a stronger condition than (2), neither (1) or (2) are sufficient condition for infinite divisibility.  In fact we will characterize infinitely divisible measures as the ones given by the central limit theorem.
 \end{rem}

The following theorem shows that the class of infinitely divisible measures is rather small.
\begin{thm}
The unique infinite divisible polynomial $p$ of degree $d$ with $m_1(p)=0$ and $m_2(p)=1$ is the one we get from the Central Limit Theorem, $p(x)=d^{-d/2} H_m(\sqrt{d}x)$.
\end{thm}

\begin{proof}
Suppose that $p$ is infinitely divisible and let $p_n$ be such that $p$ is the $n$-fold finite free convolution of $p_n$. i.e. $p=p_n^{\boxplus n}=p_n\boxplus_d \cdots \boxplus_d p_n$. Let us denote 
by $\lambda_1(p),\lambda_2(p),\ldots ,\lambda_d(p)$ the roots of $p$ where $|\lambda_1(p)|\leq|\lambda_2(p)|\leq\cdots \leq|\lambda_d(p)|$. It is readily seen that $\lambda_d^{2k}(p)\leq d m_{2k}(p)\leq d\lambda_d^{2k}(p)$. 
 Since $m_1(p)=0$, by the properties of cumulants we have that
$$\frac{1}{n}\kappa_2(p)=\kappa_2(p_n)=\frac{d}{d-1}m_2(p_n)\geq \frac{1}{d-1}\lambda_d^2(p_n).$$
Now, using the moment cumulant formula for $\kappa_4$, we obtain
$$|\kappa_4(p)|=|n\kappa_4(p_n)|=n|c_1m_4(p_n)+c_2m_2^2(p_n)|\leq nc|\lambda_d^4(p_n)+(\lambda_d^2(p_n))^2|$$
$$\leq 2nc\frac{(d-1)^2\kappa_2^2(p)}{n^2}=C\frac{\kappa_2^2(p)}{n}$$
where $c_1=\frac{d^4}{(d)_4}$, $c_2=\frac{d^4}{(d)_4}\frac{2d-3}{d-1}$, $c=\max(|c_1|,|c_2|)$ and $C=2c(d-1)^2$. 

Since $C$ and $\kappa_2^2(p)$ does not depend on $n$ we get that
$C\frac{\kappa_2^2(p)}{n} \to 0$, when $n\to\infty$, implying that $0\leq |\kappa_4(p)|\leq 0$, or $\kappa_4=0$.  Finally, since $\kappa_n$ is positive definite by Propostion \ref{divinf}, we conclude that $\kappa_n=0$, for all $n\geq 3$. This means that the only polynomial $p$ which is infinitely divisible is the one where all cumulants are null except for $\kappa_2$, or,  in other words, $p(x)=d^{-d/2} H_m(\sqrt{d}x)$.
\end{proof}

Finally, we show that $p^{\boxplus_d t}$ has real roots for $t$ large enough. This is a partial analog of a result of Nica and Speicher \cite{NS} where they prove that every probability measure on the line belongs to a $\boxplus$-partial semi-group $\{ \mu_t : t \geq 1\}$ relative to additive free convolution (i.e., $ \mu_{t+s}=\mu_t \boxplus \mu_s$ for $t, s\geq1$). Our proof resembles more the arguments in \cite{BV}.

\begin{prop}\label{large t}
Let $p\neq x^d$ be a real polynomial, then there exists $T>0$ such that for all $t>T$ the polynomial $p^{\boxplus_d t}$ has $d$ different real roots. 
\end{prop}

\begin{proof}
 By dilating and centering we may assume that $\kappa_1(p)=0$ and $\kappa_2(p)=1$. Now for any $t$ the polynomial  $D_{1/\sqrt{t}} (p^{\boxplus_d t})$ has $d$ different real roots if and only $p^{\boxplus_d t}$ has them.  In the proof of the central limit theorem we never used that $N$ was an integer and thus
$$D_{1/\sqrt{t}} (p^{\boxplus_d t})\to d^{-d/2} H_m(\sqrt{d}x).$$

Now, convergence of the cumulants is equivalent to convergence of the coefficients. Since, $d^{-d/2} H_m(\sqrt{d}x)$, has $d$ \emph{different} real roots, by continuity, $D_{1/\sqrt{t}} (p^{\boxplus_d t})$ has also $d$ different real roots  for $t$ large enough. 
\end{proof}

\begin{rem}
One may wonder if it is true that if a polynomial $p$ has real roots then also $p^{\boxplus_d t}$ has for all $t>0$ real roots. This is not true, an explicit example being a Poisson with parameter $1/4$ and degree $4$.  Indeed, $Poiss(1/4,4)=z^4-z^3$ has roots $\{0,0,0,1\}$ while the polynomial $Poiss(1/4,4)^{\boxplus_4 4/3}=z^4-(4/3) z^3+z^2/6+z/54+5/2592$ has roots $\{0.250561, 1.17721,-0.0472193 - 0.0656519 i, -0.0472193 + 0.0656519 i  \}.$

\end{rem}

As a final remark we observe that Cramer's theorem does no hold in the finite free case.

\begin{prop}(Failure of Cramer's Theorem) There exist polynomials $p^+$ and $p^-$ which are not a dilation of $d^{-d/2} H_m(\sqrt{d}x)$ such that $p^+\boxplus_d p^-= d^{-d/2} H_m(\sqrt{d}x)$.
\end{prop}
\begin{proof}
 Consider, for $\epsilon>0$, the real polynomials $p^+_\epsilon$ and $p^-_\epsilon$ (possibly with complex roots) such that for $\delta\in\{+,-\}$ the cumulants of $p^\delta$ are given by  $\kappa_1(p_\epsilon^{\delta})=0, \kappa_2(p_\epsilon^\delta)=1,\kappa_3(p^\delta)=\delta\epsilon$ and  $\kappa_n(p)=0$, for $n>0$. By continuity, since  $p^d_\epsilon$ approximate $d^{-d/2} H_m(\sqrt{d}x)$ as $\epsilon\to 0$, there exists $\epsilon>0$ such that $p^+_\epsilon$ and $p^-_\epsilon$ both have reals roots.  By construction, $D_{1/2}p^+_\epsilon\boxplus_dD_{1/2}p^+_\epsilon=p^+_\epsilon\boxplus_d p^-_\epsilon= d^{-d/2} H_m(\sqrt{d}x).$
\end{proof}

\textbf{Acknowledgment} We thank Drew Armstrong for fruitful discussions on the paper and, in particular, for pointing out Lemma $4.4$. The authors were supported by CONACYT Grant 222668 during the writing of this paper.

\end{document}